\documentclass[11 pt]{amsart}
\usepackage{amssymb}
\usepackage{amsmath}
\usepackage{amsthm}
\usepackage{latexsym,amsfonts,amssymb,mathrsfs}
\usepackage [all,2cell]{xy}
\usepackage{mathabx}
\UseAllTwocells
\numberwithin{equation}{section}

\theoremstyle{plain}   

 \newtheorem{thm}[equation]{Theorem}
 \newtheorem{cor}[equation]{Corollary}
 \newtheorem{lem}[equation]{Lemma}
 \newtheorem{prop}[equation]{Proposition}

\newtheorem*{thm*}{Theorem}
\newtheorem*{cor*}{Corollary}
\newtheorem*{lem*}{Lemma}
\newtheorem*{prop*}{Proposition}
\newtheorem*{claim*}{Claim}

\theoremstyle{definition}
\newtheorem{defn}[equation]{Definition}
\newtheorem*{defn*}{Definition}
\newtheorem{cons}[equation]{Construction}

\theoremstyle{remark}
\newtheorem{rmk}[equation]{Remark}
\newtheorem{ex}[equation]{Example}

\newtheorem*{rmk*}{Remark}
\newtheorem*{ex*}{Example}
\newtheorem*{exs*}{Examples}

\newcommand{\cx}{\mathbb{C}}
\newcommand{\A}{\mathcal{A}}
\newcommand{\BA}{\mathbb{A}}
\newcommand{\BE}{\mathbb{E}}
\newcommand{\R}{\mathcal{R}}
\newcommand{\K}{\mathcal{K}}
\newcommand{\BK}{\mathbb{K}}
\newcommand{\F}{\mathcal{F}}
\newcommand{\G}{\mathcal{G}}
\newcommand{\fH}{\mathcal{H}}
\newcommand{\fM}{\mathcal{M}}

\newcommand{\bN}{\mathbf{N}}

\newcommand{\C}{\mathcal{M}}
\newcommand{\D}{\mathcal{N}}

\newcommand{\M}[1]{M_{#1}(\R)}
\newcommand{\GL}[1]{GL_{#1}(\R)}
\newcommand{\GLH}[1]{GL_{#1}(\BE(\R))}
\newcommand{\Mod}{\fM(\R)}

\newcommand{\gC}[1]{\widehat{\C}(\un{#1})}
\newcommand{\ModH}{\fM(\BE(\R))}

\newcommand{\gD}[1]{\widetilde{\D}(\un{#1})}
\newcommand{\n}[1]{\textbf{#1}}

\newcommand{\id}[1]{\text{id}_{#1}}
\newcommand{\I}[1]{I_{#1}}
\newcommand{\Id}{\text{id}}
\newcommand{\bp}{\boxplus}
\newcommand{\op}{\oplus}
\newcommand{\Bin}[1]{\bigboxplus _{i\in {#1}}}
\newcommand{\ob}[2]{\{ {#1}_S,{#2}_{S,T}\} }
\newcommand{\Hom}[2]{\underline{Hom}(\Delta ^1,\D({#1},{#2}))}
\newcommand{\Homdd}[4]{\underline{Hom}(\D({#1},{#2}),\D({#3},{#4}))}
\newcommand{\on}[1] {\operatorname{#1}}
\newcommand{\st} {S \cup  T}
\newcommand{\stu} {S \cup  T \cup  U}
\newcommand{\tu} {T \cup  U}
\newcommand{\bst} {B_S \! \op \! B_T}
\newcommand{\bstu} {B_S \! \op\!  B_T\! \op \! B_U}
\newcommand{\btu} {B_T \! \op \! B_U}
\newcommand{\cst} {C_S \! \op \! C_T}

\newcommand{\fat}[1]{\widetilde{\D}(\underline{#1})}
\newcommand{\thin}[1]{\widehat{\D}(\underline{#1})}

\newcommand{\un}[1]{\underline{#1}}

\begin{document}
\title{Spectra associated to symmetric monoidal bicategories}

\author{Ang\'{e}lica M. Osorno}

\begin{abstract}
We show how to construct a $\Gamma$-bicategory from a symmetric monoidal bicategory, and use that to show that the classifying space is an infinite loop space upon group completion. We also show a way to relate this construction to the classic $\Gamma$-category construction for a bipermutative category. As an example, we use this machinery to construct a delooping of the $\K$-theory of a bimonoidal category as defined in \cite{BDR}.
\end{abstract}

\maketitle


\section{Introduction}

Symmetric monoidal bicategories appear in many contexts in mathematics. Some examples include $Bimod$, the bicategory of rings, bimodules and homomorphisms of bimodules, with tensor product as the monoidal structure; and $nCob$, the bicategory of closed $n$-manifolds, cobordisms and diffeomorphisms between cobordisms, with disjoint union as the monoidal structure. The definition of symmetric monoidal bicategory is cumbersome. The most concise definition can be found in \cite{schommer}. Shulman \cite{shulman} shows how to obtain symmetric monoidal bicategories from symmetric monoidal double categories, which are easier to understand.

In view of the importance of the construction of spectra from symmetric monoidal categories \cite{mayperm, segal}, we can ask then if there is a similar construction for symmetric monoidal bicategories. More specifically, we would like to know if the group completion of the classifying space of a symmetric monoidal bicategory is an infinite loop space. In this paper we show that this is the case for \emph{strict} symmetric monoidal bicategories. To do this, we use the theory of Segal's $\Gamma$-spaces \cite{segal}. This procedure gives a functor from symmetric monoidal bicategories to spectra that we call $\BA$.

Given a bimonoidal category $\R$, we construct the bicategory of $\R$-modules, $\Mod$. It turns out that the $\K$-theory of $\R$, as defined in \cite{BDR} is the group completion of the classifying space of $\Mod$. We prove that $\Mod$ is a strict symmetric monoidal bicategory, and hence, $\K(\R)=\BA _0 (\Mod)$ is an infinite loop space. We compare this result with those of \cite{BDRRst}, where the authors obtain an equivalence of spaces
\begin{equation}\label{equivk}
\K (\R) \overset{\sim}\longrightarrow K(\BE(\R)),
\end{equation}
where the right-hand side is the zeroth space of the algebraic $K$-theory spectrum of the ring spectrum $\BE(\R)$. We prove that this equivalence is an equivalence of infinite loop spaces.

\subsection{Organization}

Section \ref{back} contains the necessary background information on bicategories and Segal's $\Gamma$-spaces. We state the main results in Section \ref{roadmap}. Section \ref{algkthy} is dedicated to the application of the main results to the $K$-theory of bimonoidal categories. The remaining sections contain the main proofs and constructions. Appendix \ref{classifying} gives an account of the construction of the classifying space of a bicategory.

\subsection{Notation}

We denote the functor from the category of (small) permutative categories to the category of spectra by $\BE$. Let $R$ be a ring spectrum. Let $\C (R)$ be the category of finitely generated free modules over $R$. Then the $K$-theory spectrum of $R$ is defined as $\BK(R)=\BE(\C(R))$. The $K$-theory space is the zeroth space of this spectrum and it is denoted by $K(R)$.




\subsection{Warning!}
Both monoidal bicategories and bimonoidal categories appear in this paper. We hope that the similarity of these terms leads to no confusion.


\subsection{Acknowledgements}

The results from this paper are part of the author's Ph.D. thesis under the supervision of Mark Behrens. We would like to thank Peter May for suggesting the current organization and for his comments on earlier versions of this paper.

\section{Background}\label{back}

\subsection{Bicategories}
In this paper we will be working with bicategories. For definitions and proofs of the basic theorems, we refer the reader to  the earlier papers on bicategories \cite{benabou, streetfib} and some more recent accounts \cite{leinster, schommer}.

Throughout the document we will assume categories and bicategories are enriched over simplicial sets without explicitly saying it. A bicategory is enriched over simplicial sets if all the categories of morphisms are enriched. Whenever we work with a non-enriched category we will indicate it explicitly.

Given a bicategory $\mathcal{C}$, we will denote the vertical composition by $\circ$, the horizontal composition by $\ast$. For an object $A$ in $\mathcal{C}$, $I_A$ denotes the identity 1-morphism. The associativity isomorphism is denoted by $\alpha$, while the right and left identity isomorphisms are denoted by $r$ and $l$, respectively.

Let $F: \mathcal{C}\rightarrow \mathcal{D}$ be a pseudofunctor. We denote the functoriality 2-isomorphisms by
$$F^2_{f,g}: F(g)\ast F(f)\rightarrow F(g\ast f) \quad\text{and}\quad F^0_A :\I{FA}\rightarrow F(\I{A}).$$
These are subject to coherence axioms as stated in \cite[Section 1.1]{leinster}.

Given a transformation $\eta : F \rightarrow G$ between pseudofunctors we denote the 1-morphism by $\eta _A: FA \rightarrow GA$ and the naturality 2-isomorphism by $\eta ^2 _f : Gf\ast \eta _A \rightarrow \eta _B \ast Ff$.

We will make extensive use of pasting diagrams for bicategories. A pasting diagram is a polygonal arrangement on the plane, where the vertices correspond to objects, the directed edges correspond to 1-morphisms and the faces are usually filled with double arrows corresponding to 2-morphisms. For example, the diagram
$$\xymatrixrowsep{.5pc}\xymatrix{
& \ar[ddr]^g &\\
\rrtwocell<\omit>{\sigma}& &\\
\ar[uur]^f \ar[rr]_h &&
}$$
indicates that $\sigma$ is a 2-morphism from $g\ast f$ to $h$.

We can combine pasting diagrams to depict certain compositions of 2-morphisms. For example the diagram
$$\xymatrixrowsep{.5pc}\xymatrix{
\ar[rr]^f \ar[rdd]_g && \ar[rdd]^k &\\
\rrtwocell<\omit>{\varphi}& \rrtwocell<\omit>{\psi}&&\\
& \ar[uur]_h \ar[rr]_l &&
}
$$
represents the 2-morphism given by the composition
$$k\ast f \overset{k\ast \varphi}\Longrightarrow k\ast (h\ast g)
\overset{\alpha}\Longrightarrow (k\ast h) \ast g \overset{\psi\ast
g}\Longrightarrow l\ast g.$$

We note here that the 2-morphisms are not actually composable; we need to use the associativity isomorphism. In general, for larger diagrams, the source and target of the 2-morphisms we are composing may differ by their bracketing. By the Coherence Theorem of Bicategories \cite{MP} we know that there is a unique canonical associativity isomorphism between two bracketings, so we use this isomorphisms to connect the source and target of the 2-morphisms we are composing and hence make sense of the diagram.

Once we specify a bracketing of the outside 1-morphisms, the diagram has a unique meaning, no matter what order we use to compose the 2-morphisms. We refer the reader to \cite{KS}.

When we say ``pasting diagram $A$ is equal to pasting diagram $B$'' we mean that with a given bracketing of the outside 1-morphisms, the given 2-morphisms that they both define are equal. Note that if this is true for a given bracketing, it is true for all bracketings.

\subsection{$\Gamma$-spaces}\label{gammaspaces}

Segal's $\Gamma$-spaces give an infinite loop space machine. In \cite{segal}, it is shown that a symmetric monoidal category $\C$ gives rise to a $\Gamma$-space, and hence a cohomology theory. We will recall briefly what a $\Gamma$-space is, since the definition will play a central role in the paper.

Let $Fin_{\ast}$ denote that (skeletal) category of finite pointed sets and pointed maps. The skeletal version has as objects the sets $\underline{n}=\{0,1,\dots, n\}$, for $n\geq 0$. Here $0$ is the basepoint.

For $1\leq k\leq n$, we define $i_k:\underline{n}\rightarrow \underline{1}$ as:
$$i_k(j)=\begin{cases}
          0 & \text{if}\quad j\neq k\\
    1       & \text{if}\quad j= k.
         \end{cases}$$

\begin{defn}
A \emph{$\Gamma$-space} $X$ is a functor $X: Fin_{\ast}\rightarrow Top$. We say $X$ is \emph{special} if the map
$$P_n: X(\underline{n})\rightarrow X(\underline{1})^{\times n},$$
obtained by assembling the maps $i_k$, is a weak equivalence for all $n\geq 0$.
\end{defn}

The conditions in the definition above roughly imply that the space $X(\underline{1})$ has a multiplication that is associative and commutative up to coherent higher homotopies. The precise statement in given by the following theorem:

\begin{thm}\cite[Prop. 1.4]{segal}\label{gspaces}
Let $X$ be a special $\Gamma$-space. Then $X(\un{1})$ is an $H$-space and its group completion, $\Omega B X(\un{1})$ is an infinite loop space.
\end{thm}

Segal and May \cite{mayperm} show how to construct a $\Gamma$-category from a symmetric monoidal category, thus getting an infinite delooping of the classifying space of a symmetric monoidal category. As we mention in the introduction, given a permutative category $\C$, we denote the associated spectrum by $\BE(\C)$. We will follow a similar approach in the context of bicategories.

\section{Statement of results}\label{roadmap}


In broad terms, a symmetric monoidal bicategory is a bicategory with a product pseudofunctor that is associative, unitary, and commutative up to coherent natural equivalences. The precise definition is quite involved, as one would imagine. For a precise definition and a historical account of the theory of symmetric monoidal bicategories we refer the reader to \cite{schommer}. Shulman \cite{shulman} provides a way of constructing examples of symmetric monoidal bicategories. 

In this paper we will be working with a strict version of symmetric monoidal bicategories, so that is what we will define here.

\begin{defn}
 A \emph{strict symmetric monoidal bicategory} $(\C, \bp, 1, \beta)$ consists of the following data:

 \begin{itemize}
  \item a bicategory $\C$;
  \item a pseudofunctor of bicategories $\bp: \C \times \C \rightarrow \C$;
  \item an object 1 in $\C$ called the \emph{unit};
  \item a transformation $\beta : \bp \rightarrow \bp \circ \tau,$ where $\tau$ denotes the twist pseudofunctor $\C \times \C \rightarrow \C \times \C$.
 \end{itemize}

The monoidal product $\bp$ is required to be strictly associative, and 1 is a strict unit. The following diagrams must commute:

\begin{equation*}
 \xymatrix{
  A\bp B\bp C \ar[r]^{\beta\bp I} \ar[dr]_{\beta} & B\bp A\bp C \ar[d]^{I\bp \beta}\\
  &B\bp C\bp A
 } \qquad
\xymatrix{
  A\bp B\bp C \ar[r]^{I \bp \beta} \ar[dr]_{\beta} & A\bp C\bp C \ar[d]^{\beta\bp I}\\
  &C\bp A\bp B
 }
\end{equation*}
\begin{equation*}
\xymatrix{
  A\bp B \ar@{=}[rr] \ar[dr]_{\beta} && A\bp B \\
  &B\bp A \ar[ur]_{\beta}
 }
\end{equation*}

\end{defn}

\begin{rmk}
 The strict version defined above is just one of the many ways in which one can strictify the notion of symmetric monoidal bicategory. In fact, there are different levels of strictness one could consider. We choose this level because it is convenient to work with and it covers the applications we have in mind. It is easy to check that this definition is a special case of the general definition of symmetric monoidal bicategory. On the other hand, it is not yet known whether or not a general symmetric monoidal category can be strictified.

 Indeed, it is known that any monoidal bicategory is equivalent to a Gray monoid \cite{GPS}, which is a 2-category with a fairly strict product pseudofunctor, which is strictly associative and unital. The extra data needed for a the monoidal bicategory to be symmetric is carried across the equivalence but does not necessarily get any stricter. It is unlikely that a symmetric monoidal bicategory can be strictified any further.
\end{rmk}

The symmetric monoidal structure on a bicategory translates into an $H$-space structure on its classifying space. In the case of symmetric monoidal categories, we know that we actually obtain an infinite loop space structure upon completion. To show that this is also the case for (strict) symmetric monoidal bicategories we will be using Segal's $\Gamma$-space machine in the context of bicategories.

Let $Bicat$ denote the category of (small) bicategories and pseudofunctors.

\begin{defn}\label{defequiv}
 A pseudofunctor $\F: \mathcal{C} \rightarrow \mathcal{D}$ is an \emph{equivalence of bicategories} if there exists a pseudofunctor $\G: \mathcal{D} \rightarrow \mathcal{C}$ and natural equivalences, i.e. weakly invertible transformations,
 $$id_{\mathcal{C}} \simeq \G \circ \F \quad \text{and}\quad id_{\mathcal{D}}\simeq \G \circ \F.$$
\end{defn}

Other authors use the term \emph{biequivalence} to refer to the definition above.

\begin{defn}
A \emph{$\Gamma$-bicategory} $\A$ is a functor $\A: Fin_{\ast}\rightarrow Bicat$. We say $A$ is \emph{special} if the map
$$P_n: \A(\underline{n})\rightarrow \A(\underline{1})^{\times n}$$
is an equivalence of bicategories for all $n\geq 0$.
\end{defn}

This definition is analogous to that of a special $\Gamma$-space, with the connection made clear by the following lemma, where $|\bN (-)|$ denotes the classifying space functor as constructed in Appendix \ref{classifying}.

\begin{lem}\label{lemma}
 Let $\A$ be a special $\Gamma$-bicategory. Then $|\bN \A|:\Gamma \rightarrow Top$ is a special $\Gamma$-space.
\end{lem}

\begin{proof}
 The classifying space functor $|\bN (-)|: Bicat \rightarrow Top$ preserves products and sends equivalences of bicategories to homotopy equivalences of spaces (Proposition \ref{transhomo}).
\end{proof}

\begin{thm}\label{gamma}
 Let $\C$ be a (strict) symmetric monoidal bicategory. Then there is a special $\Gamma$-bicategory $\widehat{\C}$ such that
$$\widehat{\C}(\underline{1})\cong \C .$$
Therefore the classifying space $|\bN \C|$ is an infinite loop space upon group completion.
\end{thm}

We delay the proof of this theorem until Section \ref{proofgamma}. Given a (strict) symmetric monoidal bicategory $\C$, we denote the associated spectrum produced by the machine above by $\BA(\C)$.

We would like to compare the $\Gamma$-bicategory construction to the standard $\Gamma$-category construction. To that end, we provide a map of the corresponding $\Gamma$-spaces obtained from a symmetric monoidal bicategory and a symmetric monoidal category when these are closely related.

Let $(\D,\op, 0, \tau)$ be a permutative category enriched over simplicial spaces. We further require that for all objects $A$ and $B$, the simplicial space $\D(A,B)$ is the levelwise nerve of a simplicial category (that is, a simplicial object in the category of small categories).

In Section \ref{constructionfat}, we construct a $\Gamma$-category $\widetilde{\D}$ inspired by the construction of the $\Gamma$-bicategory and relying on the simplicial enrichment. This $\Gamma$-category is levelwise equivalent to the standard $\Gamma$-category construction \cite{mayperm}. The following theorem makes the relationship between the two constructions explicit.

\begin{thm}\label{mton}
Let $\C$ be a strict symmetric monoidal bicategory and $\D$ a permutative category enriched over simplicial categories as above. Suppose that $\C$ and $\D$ have the same set of objects, and the symmetric monoidal structures coincide. Furthermore, suppose that for all pairs of objects $(A,B)$ there is a map
$$N\C(A,B) \rightarrow \D(A,B),$$
and that these maps are compatible with the categorical and symmetric monoidal structures on $\C$ and $\D$. More precisely, this means that these maps commute with (horizontal) composition and the monoidal product, and send the identity 1-morphism to the identity morphism. Then there is a canonical map of $\Gamma$-spaces
$$|\bN\widehat{\C}|\rightarrow |N \widetilde{\D}|.$$
Therefore there is an induced map of spectra
$$\BA(\C)\rightarrow \BE(\D).$$
\end{thm}

The proof of this theorem is in Section \ref{proofmton}.

\section{Delooping $K$-theory of bimonoidal categories}\label{kthy}

In \cite{BDR}, N. Baas, B. Dundas and J. Rognes introduce the notions of $2K$-theory and 2-vector bundles as a way to categorify topological $K$-theory and vector bundles. One of their objectives is to define a cohomology theory of a geometric nature that has chromatic level 2.

In general they define the $K$-theory space of a (symmetric) bimonoidal category $\R$ in terms of a bar construction for monoidal categories (Subsection \ref{kthydef}) and construct an equivalence of spaces
\begin{equation*}\tag{\ref{equivk}}
\K (\R) \overset{\sim}\longrightarrow K(\BE(\R))
\end{equation*}
between the $K$-theory of the bimonoidal category $\R$ and the algebraic $K$-theory space of the ring spectrum $\BE(\R)$ (Subsection \ref{algkthy}).

In Subsection \ref{modr} we prove that the $K$-theory space of $\R$ can be constructed as the appropriate group completion of the classifying space the bicategory of modules over $\R$, $\Mod$, and in Subsection \ref{symm} we show that $\Mod$ is symmetric monoidal.

As such, we can use $\Mod$ as input for Theorem \ref{gamma} to obtain an infinite loop space structure on $\K(\R)$. On the other hand, the algebraic $K$-theory of the ring spectrum $\BE(\R)$ is constructed as the classifying space of the symmetric monoidal category of modules $\ModH$. With $\C=\Mod$ and $\D=\ModH$ as input for Theorem \ref{mton} we obtain a map of spectra
$$\BK(\R)=\BA(\Mod)\longrightarrow \BE(\ModH)=\BK(\BE(\R))$$
which at the level of zeroth spaces gives the equivalence in (\ref{equivk}).

\subsection{Definition of $K$-theory of a bimonoidal category}\label{kthydef}

A (symmetric) \emph{bimonoidal category} $(\R, \oplus, \otimes)$ is a category $\R$ endowed with a symmetric monoidal structure $\oplus$ and a (symmetric) monoidal structure $\otimes$, which distributes over $\oplus$ up to coherent isomorphisms. The precise definition can be found in \cite{laplaza} for the symmetric case and in \cite{EM} for the non-symmetric one, although in the latter the distributivity isomorphisms are not required to be invertible. Any symmetric bimonoidal category is equivalent to a bipermutative category \cite[Prop. VI.3.5]{bipermold} and any bimonoidal category is equivalent to a strict bimonoidal category \cite[Thm. 1.2]{guillou}.

\begin{defn}
 A \emph{strict bimonoidal category} ($\R, \oplus, 0, \gamma _{\oplus}, \otimes, 1, \delta$) is a permutative category ($\R, \oplus, 0, \gamma _{\oplus}$), together with a strict monoidal structure ($\otimes, 1$), such that right distributivity and nullity of zero hold strictly, and there is a left distributivity natural isomorphism
 \begin{equation*}
 \delta : a\otimes (b\oplus c) \rightarrow (a\otimes b)\oplus (a\otimes c).
 \end{equation*}

 These satisfy the coherence axioms spelled out in \cite[Definition 3.1]{guillou}.

\end{defn}

If we further require the product $\otimes$ to be permutative (strict symmetric monoidal), the category is called \emph{bipermutative}. Strict bimonoidal and bipermutative categories are analogues of semirings and commutative semirings, respectively.

Let $(\R, \oplus, 0, c_{\oplus}, \otimes, 1, \delta )$ be a  strict bimonoidal category. Then, as in \cite{BDR}, we can define $\M{n}$, the category of $n \times n$ matrices over $\R$. Its objects are matrices $V=(V_{i,j})_{i,j=1}^n$ whose entries are objects of $\R$. The morphisms are matrices $\phi=(\phi_{i,j})_{i,j=1}^n$ of isomorphisms in $\R$, such that the source (resp, target) of $\phi _{i,j}$ is the $(i,j)-$entry of the source (target) of $\phi$. As a category, $\M{n}$ is isomorphic to $\R ^{n \times n}$.

Moreover, $\M{n}$ is a monoidal category, with multiplication
$$\M{n}\times \M{n} \overset{\cdot}{\rightarrow} \M{n}$$
given by sending the pair $(U,V)$ to
$$W_{ik}=\bigoplus _{j=1}^n U_{ij}\otimes V_{jk}.$$

Since $\oplus$ is strictly associative, there is no ambiguity.

This multiplication has a unit object $I_n$, given by the matrix with $1$ in the diagonal and $0$ elsewhere. The objects $0$ and $1$ are strict units for $\oplus$ and $\otimes$ respectively, and the nullity of 0 holds strictly, so $I_n$ is a strict unit as well.

\begin{prop}\cite[3.3]{BDR}
Matrix multiplication makes $(\M{n}, \cdot, I_n)$ into a monoidal category.
\end{prop}

The natural associativity isomorphism
$$\alpha : U\cdot (V\cdot W) \rightarrow (U\cdot V)\cdot W$$
is given by entry-wise use of $c_{\oplus}$ and $\delta$.

Recall that if $R$ is a semi-ring, $GL_n(R)$ is the  subgroup of $M_n(R)$ that contains all the matrices whose determinant is a unit in the ring completion $Gr_+(R)$. The following definition is also taken from \cite{BDR}.

\begin{defn}
 Let $\GL{n}\subset \M{n}$ be the full subcategory of matrices $V=(V_{i,j})_{i,j=1}^n$ whose matrix of path components lies in $GL_n(\pi_0(\R))$. We call $\GL{n}$ the category of \emph{weakly invertible matrices}. By convention we will let $\GL{0}=\underline{1}$ be the unit category, with one object and one morphism.
\end{defn}

Note that $\GL{n}$ inherits a monoidal structure from $\M{n}$.

Given a monoidal category $\fM$, the authors  in \cite{BDR} define a bar construction for monoidal categories, $B\fM$, which is a simplicial object in $Cat.$  As pointed out in Remark \ref{barsegal}, this definition coincides with the Segal nerve of the bicategory $\Sigma \fM$, that is, the bicategory with one object whose category of morphisms is given by $\fM$.

We note that block sum of matrices in $\R$ makes
$$\coprod _{n\geq 0} |B\GL{n}|$$
into an $H$-space, and hence we define the $K$-theory of $\R$, as the group completion
$$\K(\R):=\Omega B \bigl(\coprod _{n\geq 0} |B\GL{n}| \bigr).$$

The motivation behind the definition of $K$-theory for bimonoidal categories comes from the categorification of complex $K$-theory. As we know well, the complex $K$-theory space classifies virtual vector bundles.

A 2-vector space, as defined in \cite{KV}, is a category equivalent to $(Vect_{\cx})^n$ for some $n$. Heuristically, this should be thought of as a module category over $Vect_{\cx}$. In \cite{BDR}, the authors introduce the notion of a complex 2-vector bundle over a topological space and construct a classifying space for these bundles. A 2-vector bundle is roughly a bundle of 2-vector spaces over $X$, defined in terms of transition functions, which are given by matrices of vector spaces. For the precise definition we refer the reader to \cite[Section 2]{BDR}.

One of the main results in \cite{BDR} is that the stable equivalence classes of virtual 2-vector bundles over a space $X$ are in one-to-one correspondence with homotopy classes of maps from $X$ to  $\K(Vect_{\cx})$, where $Vect_{\cx}$ is a considered as a bipermutative category using direct sum and tensor product.

\subsection{Relationship with algebraic $K$-theory}\label{algkthy}

Given a strict bimonoidal category $\R$, by forgetting the multiplicative structure, we can construct the spectrum $\BE(\R)$ associated to the permutative category $(\R, \oplus)$. The results of \cite{EM} and \cite{bipermnew} show that the multiplicative structure of $\R$ makes $\BE(\R)$ into a ring spectrum, and furthermore, if $\R$ is bipermutative, $\BE(\R)$ is an $E_{\infty}$ ring spectrum. Note that in \cite{BDR, BDRRst} the authors denote $\BE(\R)$ by $H\R$, pointing out the analogy to the Eilenberg-MacLane spectrum of a ring.

The natural inclusion $B\R \rightarrow \BE_0(\R) $ extends to a map $|B\GL{n}|\rightarrow BGL_n(\BE(\R))$. Taking the group completion of the disjoint union over $n$, we get the map
\begin{equation}\label{maineq}
\K (\R) \overset{\sim}\longrightarrow K(\BE(\R)),
\end{equation}
of (\ref{equivk}), which \cite{BDRRst} proves to be an equivalence of spaces. Here the
right-hand side is the zeroth space of $\BK(\BE(\R))$, which we view as the
algebraic $K$-theory ring spectrum associated to the ring spectrum $\BE(\R)$.

We are headed towards an intrinsic description of the left-hand side as
an infinite loop space in Theorem \ref{infloop} and a proof that the equivalence is
compatible with these infinite loop space structures in Corollary \ref{maincor}.

\subsection{$K$-theory as a classifying space of a bicategory}\label{modr}

\begin{defn}
Let $\Mod$ be the bicategory of finite dimensional free modules over $\R$, defined as follows. The objects are labeled by the natural numbers $\n{n} \geq 0$. Given objects $\n{n},\n{m}$, the category of morphisms is
$$\Mod(\n{n},\n{m})=
 \begin{cases}
    \GL{n} & \text{if}\quad n=m\\
    \emptyset       & \text{if}\quad n\neq m.
  \end{cases}$$
and the composition is given by matrix multiplication. In other words,
$$\Mod = \coprod _{n\geq 0} \Sigma \GL{n}.$$
\end{defn}


\begin{ex}
 Let $Vect_k$ be the bipermutative category of vector spaces over the field $k$. Then $Mod_{Vect_k}$ is a sub-bicategory of the bicategory of 2-vector spaces defined by Kapranov and Voevodsky \cite{KV}. The 1-morphisms are matrices of vector spaces such that their matrices of dimensions have determinant $\pm 1$.
\end{ex}

We can use the bicategory $\Mod$ to give an alternative definition of the $K$-theory of $\R$. We have that
$$\coprod _{n\geq 0} |B\GL{n}|=\coprod _{n\geq 0} \bN \Sigma \GL{n} = \bN \bigl( \coprod _{n\geq 0} \Sigma \GL{n} \bigr) = \bN \Mod,$$
where $\bN(-)$ denotes the Segal nerve (see Appendix \ref{classifying}). Hence, we can describe the $K$-theory space as $\Omega B |\bN \Mod|$.

This is the definition we will use in the following sections. Furthermore, we will show that the $H$-space structure comes from a pseudofunctor
$$\Mod \times \Mod \rightarrow \Mod,$$
which will give $\Mod$ the structure of a symmetric monoidal bicategory.

\subsection{Symmetric monoidal structure on $\Mod$}\label{symm}

Just as we can take direct sum of modules over a ring, we can take direct sum of modules over a bimonoidal category. This will provide $\Mod$ with a symmetric monoidal structure, which in turn will give rise to an infinite delooping of $\K(\R)$.

\begin{thm}
The bicategory $\Mod$ is strict symmetric monoidal with the monoidal operation given by block sum of matrices:
\begin{align*}
 \bp: \Mod \times \Mod &\rightarrow \Mod\\
 (\n{n},\n{m})&\mapsto \n{n+m}\\
 (U,V)&\mapsto \left[ \begin{array}{c|c}
     U & 0\\
\hline
0 & V\\
    \end{array}
\right]\\
 (\varphi, \psi) &\mapsto \left[ \begin{array}{c|c}
     \varphi & 0\\
\hline
0 & \psi\\
    \end{array}
\right].
\end{align*}
The matrix $[0]$ is the matrix with all entries equal to $0$, the unit of $\op$ in $\R$.

\end{thm}

\begin{proof}
 We first note that the operation described above gives a strict pseudofunctor of bicategories, since it preserves the identity and the composition:
 $$I_n \bp I_m = I_{n+m},$$
 $$\left[ \begin{array}{c|c}
     U' & 0\\
\hline
0 & V'\\
    \end{array}
\right] \ast \left[ \begin{array}{c|c}
     U & 0\\
\hline
0 & V\\
    \end{array}
\right]=\left[ \begin{array}{c|c}
     U'\ast U & 0\\
\hline
0 & V'\ast V\\
    \end{array}
\right].
$$

The second equation holds because of the strict nullity and unity of 0 in $\R$.

The unit of $\bp$ is $\n{0}$. We note that for $U\in \GL{n}$, $V\in \GL{m}$, and $W\in \GL{p}$:
$$(U\boxplus V)\boxplus W = U \boxplus (V \boxplus W),$$
$$I_0\boxplus U=U=U\boxplus I_0.$$

The natural equivalence $\beta _{n,m}: \n{n} \bp \n{m} \rightarrow \n{m} \bp \n{n}$  is given by the block matrix
$$\left[ \begin{array}{c|c}
     0 & I_m\\
\hline
I_n & 0\\
    \end{array}
\right].$$

Since $0$ and $1$ are strict units in $\R$, for $U\in \GL{n}$ and $V\in \GL{m}$,
$$\beta _{n,m}\ast (U\boxplus V)=\left[ \begin{array}{c|c}
     0 & V\\
\hline
U & 0\\
    \end{array}
\right]
=(V\boxplus U) \ast \beta_{n,m},$$
so $\beta$ is a strict transformation.

We note that $\beta _{m,n}\ast \beta_{n,m} = I_{n+m}$ which both implies that $\beta$ is a natural isomorphism and that it is its self-inverse. We conclude that $\Mod$ is a strict symmetric monoidal bicategory.
\end{proof}

This, together with Theorem \ref{gamma} proves the following theorem.

\begin{thm}\label{infloop}
The $K$-theory space of the bimonoidal category $\R$ is an infinite loop space, with the additive structure provided by the block sum of matrices. More precisely, we can identify $\K(\R)$ with the zeroth space of the spectrum $\BK(\R):=\BA(\Mod)$.
\end{thm}

Now we can look back at equation (\ref{maineq}). The right hand side is an infinite loop space, with the structure given by the $\Gamma$-space construction on the classifying space of the symmetric monoidal category $\ModH$ of finitely generated free modules over $\BE(\R)$. In order to prove that the map in equation (\ref{maineq}) is a map of infinite loop spaces we will prove that the map extends to a map of $\Gamma$-spaces. In particular, we prove the following theorem.

\begin{thm}\label{zigzag}
 There is a zigzag of maps of $\Gamma$-spaces
 $$|\bN\widehat{\Mod}| \longrightarrow |N\widetilde{\ModH}| \overset{\sim}\longleftarrow |N\widehat{\ModH}|.$$
 At level 1, the right-hand map is an equality, and the left-hand map corresponds to the map in equation (\ref{maineq}).
\end{thm}

\begin{cor}\label{maincor}
There is a zigzag of equivalences of spectra
$$\BK(\R) \overset{\sim}\longrightarrow \widetilde{\BK}(\BE(\R)) \overset{\sim}\longleftarrow \BK(\BE(\R)),$$
which at the level of zeroth spaces gives the maps
$$\K(\R) \underset{(\ref{maineq})}\longrightarrow \widetilde{K}(\BE(\R)) \underset{=}\longleftarrow K(\BE(\R)).$$
\end{cor}

The model we are taking for $\BE(\R)$ is that of \cite{EM}. We construct $\ModH$ as follows.

Let $\GLH{n}$ be the group-like monoid of weakly invertible matrices over $\BE(\R)$. It is defined by the pullback
\begin{equation*}
\xymatrix{
\GLH{n} \ar[r]\ar[d]&\on{hocolim}_{\textbf{m}\in I}\Omega ^{m}M_n(\BE(\R)(m))\ar[d] \\
GL_n(\pi _0 \BE(\R)) \ar[r]& M_n(\pi _0 \BE(\R)),
}
\end{equation*}

The category $\ModH$ has as objects the natural numbers $\n{n}$. The space of morphisms is given by
$$\ModH (\n{n},\n{m})=
 \begin{cases}
    \GLH{n} & \text{if}\quad n=m\\
    \emptyset       & \text{if}\quad n\neq m.
  \end{cases}$$

Note that since $\BE(\R)(0)=N\R$, there is a map of spaces
$$N\GL{n}\longrightarrow \GLH{n}.$$

\begin{proof}[Proof of Theorem \ref{zigzag}]

 In the construction of the ring spectrum $\BE(\R)$ in \cite{EM}, the spaces of the spectrum are nerves of simplicially enriched categories. This implies that $\GLH{n}$ is the nerve of a simplicially enriched category, and thus also of a simplicial category. Thus $\ModH$ satisfies the conditions for the construction and theorems of Sections \ref{constructionfat} and \ref{mton}.

Furthermore, since the map $N\GL{n}\rightarrow \GLH{n}$ is compatible with the block sum of matrices, we can let $\C$ and $\D$ be $\Mod$ and $\ModH$ respectively in the statement of Theorem \ref{mton}. This proves the theorem.
\end{proof}

\section{Proof of Theorem \ref{gamma}}\label{proofgamma}


We will first construct the bicategory $\gC{n}$ for $n\geq 0$ as follows:

\begin{enumerate}
 \item Objects are of the form $\ob{A}{a} _{S,T}$, where $S$ runs over all the subsets of $\underline{n}$ that do not contain the basepoint 0; $(S,T)$ runs over all pairs of such subsets such that $S\cap T=\emptyset$; $A_S \in Ob \C$ and $a_{S,T} : A_{S\cup T}\rightarrow A_S \boxplus A_T$ is a 1-equivalence, that is a 1-morphism that is invertible up to isomorphism. We require further

\begin{enumerate}
\item \label{axiomobj1} $A_{\emptyset}=0$;
\item $a_{\emptyset,S}=\I{A_S}=a_{S,\emptyset}$;
\item for every triple $(S,T,U)$ of subsets such that $S\cap T=S\cap U=T\cap U=\emptyset$, the diagram
\begin{equation}\label{astu}
\xymatrixcolsep{5pc}\xymatrixrowsep{4pc}\xymatrix{
A_{S\cup T\cup U} \ar[d]_{a_{S\cup T,U}} \ar[r]^{a _{S,T\cup U}} &A_S\boxplus A_{T\cup U} \ar[d]^{\I{A_S}\boxplus a_{T,U}}\\
A_{S\cup T}\boxplus A_U \ar[r]_{a _{S,T}\boxplus \I{A_U}} &A_S \boxplus A_T \boxplus A_U}
\end{equation}
strictly commutes;

\item \label{axiomobj4} for every pair of subsets $(S,T)$, the diagram
\begin{equation}\label{ats}
\xymatrix{
A_{S\cup T} \ar@{=}[d] \ar[r]^{a_{S,T}}&A_S\boxplus A_T \ar[d]^{\beta_{A_S,A_T}}\\
A_{T\cup S} \ar[r]_{a _{T,S}}&A_T\boxplus A_S
}
\end{equation}
strictly commutes.

\end{enumerate}
\item A 1-morphism between $\ob {A}{a}$ and $\ob {A'}{a'}$ is given by a system $\ob {f}{\phi} _{S,T}$, where $S,T$ are as above; $f_S: A_S \rightarrow A' _S$ is a 1-morphism in $\C$ and $\phi _{S,T}$ is a 2-isomorphism:
$$\xymatrix{
A_{S\cup T} \ar[r]^{a _{S,T}} \ar[d]_{f_{S\cup T}}   \drtwocell<\omit>{<0>\quad \phi _{S,T}} &A_S\boxplus A_T \ar[d]^{f_S\boxplus f_T} \\
A'_{S\cup T} \ar[r]_{a'_{S,T}}&A'_S\boxplus A'_T.
}$$

We require:
\begin{enumerate}
 \item $\phi_{\emptyset,S}: f_S \ast I_{A_S}=f_S \Rightarrow f_S=I_{A_S}\ast f_S$ is the identity 2-morphism and similarly for $\phi_{S,\emptyset}$;
\item for every pairwise disjoint $S,T,U$ the following equation holds
\begin{equation}\label{phistu}
\xymatrixrowsep{3pc}\xymatrix{
A_{S\cup T\cup U} \ar[r]^-{a _{S,T\cup U}} \ar[d]_{f_{S\cup T\cup U}}   \drtwocell<\omit>{\qquad \phi _{S,T\cup U}} &A_S\bp A_{T\cup U} \ar[d]|(0.4){\quad f_S\bp f_{T \cup U}} \ar[r]^-{\I{A_S}\bp a _{T, U}} \drtwocell<\omit>{\qquad \qquad l^{-1}\ast r \bp \phi } & **[r]A_S\bp A_T\bp A_U \ar[d]^{f_S\bp f_T\bp f_U} \\
A'_{S\cup T \cup U} \ar[r]_-{a'_{S,T\cup U}}&A'_S\bp A'_{T\cup U} \ar@{}[d]|{\parallel} \ar[r]_-{\I{A'_S}\bp a'_{T,U}} & **[r]A'_S\bp A'_T\bp A'_U\\
A_{S\cup T\cup U} \ar[r]^-{a _{S\cup T, U}} \ar[d]_{f_{S\cup T\cup U}}   \drtwocell<\omit>{\qquad \phi _{S\cup T, U}} &A_{S\cup T}\bp A_U \ar[d]|(0.4){\quad f_{S\cup T}\bp f_U} \ar[r]^-{a_{S, T}\bp\I{A_U}} \drtwocell<\omit>{\qquad \quad \phi \bp l^{-1}\ast r} & **[r]A_S\bp A_T\bp A_U \ar[d]^{f_S\bp f_T\bp f_U} \\
A'_{S\cup T \cup U} \ar[r]_-{a'_{S\cup T,U}}&A'_{S\cup T}\bp A'_U \ar[r]_-{a'_{S,T}\bp \I{A'_U}} & **[r]A'_S\bp A'_T\bp A'_U,
}
\end{equation}

where $r,l$ denote the coherent identity isomorphisms in $\C$, that is $l_f:I_B \ast f \rightarrow f$ and $r_f:f\ast I_A \rightarrow f$;

\item for every $S,T$ the following equation holds:
\begin{equation}\label{phits}
\xymatrix{
A_{S\cup T} \ar[r]^{a_{S,T}} \ar[d]_{f_{S\cup T}} \drtwocell<10><\omit>{\quad \phi _{S, T}} &A_S\bp A_T \ar[d]|{\quad f_S\bp f_T} \ar[r]^{\beta}  \drtwocell<10><\omit>{\qquad \beta_{f_S,f_T}}& A_T\bp A_S \ar[d]^{f_T\bp f_S} \ar@{}[drr]|{=} &&A_{T\cup S} \ar[r]^{a_{T,S}} \ar[d]_{f_{T\cup S}} \drtwocell<\omit>{\quad \phi _{T, S}} & A_T\bp A_S \ar[d]^{f_T\bp f_S} \\
A'_{S\cup T} \ar[r]_{a'_{S,T}} &A'_S\bp A'_T \ar[r]_{\beta} & A'_T\bp A'_S &&A'_{T\cup S} \ar[r]_{a'_{T,S}} &A'_T\bp A'_S .}
\end{equation}

\end{enumerate}

\item Given 1-morphisms $\{ f_S, \phi _{S,T}\} ,\{ g_S, \gamma _{S,T}\} : \{ A_S, a_{S,T}\} \rightarrow \{ A'_S, a' _{S,T}\}$, a 2-morphism between them is given by a system $\{ \psi _S\}$ of 2-morphisms in $\C$, $\psi _S :f_S \Rightarrow g_S$, such that for all $S,T$ as above the following equation holds:
\begin{equation}\label{bst}
\xymatrixrowsep{4pc}\xymatrix{
A_{S\cup T} \ar[rr]^{a_{S,T}} \dtwocell<7>_{g_{S\cup T \quad}}^{\quad f_{S\cup T}}{\psi _{S\cup T}} \drrtwocell<\omit>{\quad \phi _{S,T}}& & A_S\bp A_T \ar[d]^{f_S\bp f_T} \ar@{}[drr]|{=} && A_{S\cup T} \ar[rr]^{a_{S,T}} \ar[d]_{g_{S\cup T}} \drrtwocell<\omit>{<2>\gamma _{S,T} \qquad \quad }& & A_S\bp A_T \dtwocell<7>_{g_S\bp g_T \qquad }^{\qquad f_S\bp f_T}{\psi _S \bp \psi _T}\\
A'_{S\cup T} \ar[rr]_{a'_{S,T}} & & A'_S\bp A'_T &&
A'_{S\cup T} \ar[rr]_{a'_{S,T}} & & A'_S\bp A'_T
}
\end{equation}

\end{enumerate}

We now need to show that these data indeed define a bicategory. We will first show that given objects $\ob {A}{a}$, $\ob{A'}{a'}$, the 1-morphisms and 2-morphisms form a category $\gC{n}(\ob{A}{a}, \ob{A'}{a'}).$

Vertical composition of 2-morphisms $\{\psi _S\}, \{\psi '_S\}$ is defined componentwise. We show that this composition satisfies equation (\ref{bst}). Given
$$\xymatrixcolsep{2pc}\xymatrix{
**[l]\{ A_S, a_{S,T}\} \rlowertwocell<10>^{\{ f, \phi \}}{\quad\{ \psi \}}  \ar[r]_(0.3){\{ g, \gamma \}} \rlowertwocell<-10>_{\{ h, \eta \}}{\quad \{ \psi ' \}} & **[r]\{ A'_S, a'_{S,T}\}}$$
we see that for all $S,T$:
$$\xymatrixrowsep{4pc}\xymatrix{
A_{S\cup T} \ar[r]^{a_{S,T}} \dlowertwocell<-7>_{h}{\psi '} \duppertwocell<7>^f{\psi} \ar[d]|(0.3)g \drtwocell<\omit>{\phi } 
& A_S\bp A_T \ar[d]^{f\bp f} \ar@{}[drr]|{=}                                                                                 
&                                                                                                                            
&A_{S\cup T} \ar[rr]^{a_{S,T}} \dtwocell<5>_h^g{\psi '}   \drrtwocell<\omit>{<2>\gamma }                                     
&                                                                                                                            
& A_S\bp A_T \dtwocell<5>_{g\bp g \quad}^{\quad f\bp f}{\psi \bp \psi} \ar@{}[drr]|{=}\\                                     
A'_{S\cup T} \ar[r]_{a'_{S,T}}                                                                                               
& A'_S\bp A'_T                                                                                                               
&                                                                                                                            
& A'_{S\cup T} \ar[rr]_{a'_{S,T}}                                                                                            
&
& A'_S\bp A'_T
&
&}
$$
$$\xymatrixrowsep{4pc}\xymatrix{
A_{S\cup T} \ar[rr]^{a_{S,T}} \ar[d]_h \drrtwocell<\omit>{<2>\eta }& & A_S\bp A_T \dlowertwocell<-10>_{h\bp h \quad}{\psi ' \bp \psi '} \ar[d]|(0.3){g\bp g} \duppertwocell<10>^{\quad f\bp f}{\psi \bp \psi}\\
A'_{S\cup T} \ar[rr]_{a'_{S,T}} & & A'_S\bp A'_T}
$$
as wanted.

We also note that $\{ \id{f_S} \}$ is a well-defined automorphism for $\{ f_S, \phi _{S,T}\}$ and it is the identity of the componentwise composition.

The composition functor $\ast$ is given by:
\begin{align*}
(\ob{g}{\gamma},\ob{f}{\phi})&\mapsto \{ g_S \ast f_S, (\gamma \diamond \phi)_{S,T} \}\\
(\{ \psi '_S\} ,\{ \psi _S\}) &\mapsto \{ \psi '_S\ast \psi _S\},
\end{align*}
where the 2-morphism $(\gamma \diamond \phi)_{S,T}$ is defined by the pasting diagram:
$$\xymatrixcolsep{.5pc}\xymatrix{
A_{S\cup T} \ar[d]_{f_{S\cup T}} \ar[rrrr]^{a_{S,T}} \drrtwocell<\omit>{\quad \phi _{S,T}} &&& \ddtwocell<\omit>&A_S\boxplus A_T \ar[dll]^{f_S\boxplus f_T} \ar[dd]^{g_S \ast f_S \bp g_T \ast f_T} \\
A'_{S\cup T} \ar[d]_{g_{S\cup T}} \ar[rr]^{a'_{S,T}} \drrtwocell<\omit>{\quad \gamma _{S,T}}&&A'_S\boxplus A'_T \ar[drr]^{g_S\boxplus g_T} \\
A''_{S\cup T} \ar[rrrr]_{a''_{S,T}}&&&&A''_S\boxplus A''_T
.}$$

The unmarked 2-isomorphism is the coherent 2-morphism corresponding to the weak functoriality of $\bp$.

Showing that $\{ g_S \ast f_S, (\gamma \diamond \phi)_{S,T} \}$ is a well-defined 1-morphism (that is, it satisfies equations (\ref{phistu}) and (\ref{phits})) can be done again using pasting diagrams, the coherence of the functoriality isomorphism and the fact that both $\ob{f}{\phi}$ and $\ob{g}{\gamma}$ satisfy those same equations. Analogously we can show that $\{ \psi '_S \ast \psi _S\}$ is a well-defined 2-morphism.

The natural associativity isomorphism in this bicategory is given by the componentwise associativity isomorphisms in $\C$. More precisely, given composable 1-morphisms $\{ f_S, \phi _{S,T}\}$, $\{ g_S, \gamma _{S,T}\}$, and $\{ h_S, \eta _{S,T}\}$, we define the 2-morphism $\{ \alpha _S \}$, where
$$\alpha _S: h_S\ast(g_S\ast f_S) \Rightarrow (h_S\ast g_S)\ast f_S$$
is the associativity isomorphism in $\C$.

The fact that
$$\{ (h_S\ast g_S)\ast f_S, ((\eta \diamond \gamma)\diamond \phi )_{S,T}\} \Rightarrow \{ h_S\ast (g_S\ast f_S), (\eta \diamond (\gamma \diamond \phi))_{S,T}\}$$
is an allowed 2-morphism in $\gC{n}$ will follow from the uniqueness of pasting diagrams and the functoriality of $\bp$. Naturality and the pentagonal axiom follow from those in $\C$.

Given an object $\{ A_S, a_{S,T}\}$, the identity 1-morphism is given by $\{ \I{A_S}, \iota _{S,T}\}$, where $\iota$ is the appropriate coherent 2-isomorphism obtained by the composition of instances of the 2-morphisms $\bp ^0$ and $r^{-1}\ast l$. It is clear that this is an allowed 1-morphism in $\gC{n}$ and that it is a weak identity, with right and left identity 2-isomorphisms given by $\{ r_{f_S}\}$, $\{l_{f_S}\}$. We conclude thus that $\gC{n}$ is indeed a bicategory.

We now need to prove that this construction extends to a functor $\widehat{\C}:Fin_{\ast}\rightarrow Bicat$. Given a morphism $\theta : \underline{n} \rightarrow \underline{m}$ in $Fin_{\ast}$ we define a pseudofunctor
$$\theta _{\ast} :\gC{n}\rightarrow \gC{m}$$
as follows:
\begin{align*}
 \{ A_S, a_{S,T}\} &\longmapsto \{ A^{\theta}_U, a^{\theta}_{U,V}\}=\{ A_{\theta^{-1}(U)}, a_ {\theta^{-1}(U), \theta^{-1}(V)}\}\\
\{ f_S, \phi _{S,T}\} &\longmapsto \{ f^{\theta}_U, \phi ^{\theta}_{U,V}\}=\{ f_{\theta^{-1}(U)}, \phi_ {\theta^{-1}(U), \theta^{-1}(V)}\}\\
\{ \psi _S\} &\longmapsto \{ \psi ^{\theta}_U\}=\{ \psi _{\theta^{-1}(U)}\},\\
\end{align*}
where $U,V$ range over disjoint subsets of $\underline{m}$ that do not contain the basepoint. Since $\theta$ is basepoint preserving, $\theta ^{-1} (U)$ does not contain the basepoint and it is an allowed indexing subset of $\un{n}$. Also, since $U$ and $V$ are disjoint, their pre-images under $\theta$ are also disjoint.

This assignment commutes strictly with all the compositions and identities in $\gC{n}$ and $\gC{m}$, giving a pseudofunctor between these bicategories.

It is clear from the construction that $\gC{1}$ is isomorphic to $\C$.

We will end the proof by showing that for every $n\geq 0$, the pseudofunctor
$$p_n: \gC{n}\rightarrow \C ^{\times n}$$
is an equivalence of bicategories (Definition \ref{defequiv}. This will show that the $\Gamma$-bicategory is special. For ease of notation we will denote the subset $\{ i\} \in \underline{n}$ as $i$. The pseudofunctor $p_n$ takes
\begin{align*}
  \{ A_S, a_{S,T}\} &\longmapsto \{ A_i\} _{i=1}^n\\
\{ f_S, \phi _{S,T}\} &\longmapsto \{ f_i\} _{i=1}^n\\
\{ \psi _S\} &\longmapsto \{ \psi _i\} _{i=1}^n.
\end{align*}

We will define an inverse pseudofunctor $i_n :\C ^{\times n} \rightarrow \gC{n}$:
\begin{align*}
\{ A_i\} _{i=1}^n &\longmapsto \{ \Bin{S} A_i, e_{S,T}\}\\
\{ f_i\} _{i=1}^n &\longmapsto \{ \Bin{S} f_i, \epsilon _{S,T} \}\\
\{ \psi _i\} _{i=1}^n &\longmapsto \{ \Bin{S} \psi _i \}.
\end{align*}

Here, $\boxplus _{i\in S}$ denotes the iterated monoidal operation $\boxplus$ with the usual order of the indices in $S\subset \underline{n}$. Recall that $\boxplus$ is strictly associative,

The 1-morphism
$$e_{S,T}: \Bin{S\cup T} A_i \longrightarrow \Bin{S}A_i \bp \Bin{T}A_i$$
is the unique composition of instances of the braiding $\beta$ that reorders the summands. It is clear that $\{ \Bin{S} A_i, e_{S,T}\}$ satisfies equations (\ref{astu}) and (\ref{ats}).

The 2-isomorphism
$$\epsilon_{S,T}: (\Bin{S} f_i \bp \Bin{T}f_i)\ast e_{S,T}\rightarrow e_{S,T}\ast (\Bin{S\cup T}f_i)$$
is given by pasting the appropriate instances of $\beta ^2 _{f_i, f_j}$, that is, the coherent 2-isomorphism of the braiding transformation (see \cite[Section 1.2]{leinster}). The coherence condition for these morphisms implies that the collection  $\{ \Bin{S} f_i, \epsilon _{S,T} \}$ satisfies  equations (\ref{phistu}) and (\ref{phits}). It is also automatic that for any $\{ \psi _i \}: \{ f_i \} \Rightarrow \{ g_i \}$, we get that $\{ \Bin{S} \psi _i\} $ is an allowed 2-morphism between $\{ \Bin{S} f_i, \Id \}$ and $\{ \Bin{S} g_i, \Id \}$.

This assignment gives a pseudofunctor with
$$i_n^2: i_n(\{ g_i\} )\ast i_n(\{ f_i\} ) \rightarrow = i_n(\{ g_i\ast f_i\} )$$
and
$$i_n^0: \I{i_n(\{ A_i\} )} \rightarrow i_n(\{ \I{A_i}\} )$$
given by the appropriate composition of instances of $\bp ^2$ and $\bp ^0$ respectively.

Clearly $p_n \circ i_n=Id_{\C ^{\times n}}$. We now construct a natural equivalence
$$\xi: Id_{\gC{n}}\rightarrow i_n \circ p_n.$$

Recall that a transformation $\eta$ between pseudofunctors $F,G: \mathcal{C} \rightarrow \mathcal{D}$ consists of a 1-morphism $\eta _A: \F A \rightarrow \G A$ for each $A\in Ob(\mathcal{C})$, and for every pair $A,B$, natural 2-isomorphisms
$$\eta ^2 _f: \G f \ast \eta _A \Rightarrow \eta _B \ast \F f.$$
The latter must satisfy some coherence conditions (see \cite[Section 1.2]{leinster}). A transformation is a natural equivalence if and only if the 1-morphism corresponding to each object is a 1-equivalence.

Hence, to construct the natural equivalence $\xi$, we need a 1-equivalence
$$\xi_{\{ A_S, a_{S,T}\}}: \{ A_S, a_{S,T} \} \rightarrow \{ \Bin{S}A_i, e_{S,T}\}$$
for every object $\{ A_S,a_{S,T}\}$ in $\gC{n}$.

Given the subset $S$, we define $a^S$ inductively as the composition:
$$A_S \xrightarrow{a_{j,S-j}} A_j \bp A_{S-j} \xrightarrow{\id{A_j}\bp a^{S-j}} A_j\bp \Bin{S-j}A_i=\Bin{S}A_i,$$
where $j$ is the smallest index in $S$.

Note that by conditions (\ref{astu}) and (\ref{ats}) on the $a_{S,T}$, the two compositions in the diagram below differ by a specified associativity 2-isomorphism:
$$\xymatrix{
A_{S\cup T} \ar[r]^-{a _{S,T}} \ar[d]_-{a^{S\cup T}}   \drtwocell<\omit>{<0>\quad \eta _{S,T}} &A_S\boxplus A_T \ar[d]^-{a^S\boxplus a^T} \\
\Bin{S\cup T}A_i \ar@<1ex>[r]_-{e_{S,T}}& **[r]\Bin{S}A_i\boxplus \Bin{T}A_i.
}$$

Since associativity isomorphisms are unique, $\{ a^S, \eta _{S,T}\}$ is a well-defined 1-morphism in $\gC{n}$. This will be the corresponding 1-morphism of the transformation $\xi$.

To complete the data of the transformation, for every pair of objects $\{ A _S, a_{S,T}\}$, $\{ A'_S, a'_{S,T} \}$ in $\gC{n}$ we need to provide a natural isomorphism $\xi ^2$, which on the component $\{ f_S,\phi _{S,T} \}$ is given by a 2-morphism
$$\xi^2(\{ f_S, \phi _{S,T} \} ): \{ \Bin{S}f_i, \epsilon _S,T \}\ast \{ a^S, \eta _{S,T} \} \Rightarrow \{ a'^S, \eta '_{S,T}\}\ast \{ f_S, \phi _{S,T}\}.$$

Given $S$, we define a 2-isomorphism in $\C$, $\phi ^S: (\Bin{S}f_i)\ast a^S \Rightarrow a'^S\ast  f_S$, inductively as the pasting diagram:
$$\entrymodifiers={+!!<0pt,\fontdimen22\textfont2>}
\xymatrixrowsep{4pc}\xymatrixcolsep{4pc}\xymatrix{
**[l]A_S \ar[r]^-{a _{j, S-j}} \ar[d]_-{f_S}   \drtwocell<\omit>{\qquad \phi _{j, S-j}} &A_j\bp A_{S-j} \ar[d]|(0.4){\quad f_{j}\bp f_{S-j}} \ar[r]^-{\I{A_j}\bp a'^{S-j}} \drtwocell<\omit>{\quad \qquad l^-1\ast r \bp \phi ^{S-j}} & **[r]\Bin{S} A_i \ar[d]^-{\Bin{S}f_i} \\
**[l]A'_S \ar[r]_-{a' _{j, S-j}} &A'_j\bp A'_{S-j} \ar[r]_-{\I{A'_j}\bp a'^{S-j}} & **[r]\Bin{S} A'_i ,
}$$
where $j$ is the smallest index in $S$. We need to show that $\{ \phi ^S\} _S$ gives a 2-morphism in $\gC{n}$, that is, it satisfies equation (\ref{bst}). This is done by induction on $|S\cup T|$ using pasting diagrams. We let $\xi ^2(\{ f_S, \phi _{S,T}\})=\{ \phi ^S \}$.

%
To show the naturality of $\xi ^2$, we need to show that
$$\entrymodifiers={+!!<0pt,\fontdimen22\textfont2>}
\xymatrixrowsep{4pc}\xymatrix{
A_S \ar[rr]^-{a^S} \dtwocell<7>_{g_S}^{f_S}{\psi _S} \drrtwocell<\omit>{\quad \phi ^S}& & \Bin{S} A_i \ar[d]^-{\Bin{S}f_i} \ar@{}[drr]|{=}
&& A_S \ar[rr]^-{a^S} \ar[d]_-{g_S} \drrtwocell<\omit>{<2>\gamma ^S \qquad \quad }& & \Bin{S} A_i \dtwocell<7>_{\Bin{S}g_i \qquad }^{\qquad \Bin{S}f_i}{\Bin{S} \psi _i}\\
A'_S\ar[rr]_-{a'^S} & & \Bin{S}A'_i
&&
A'_S \ar[rr]_-{a'^S} & & \Bin{S}A'_i.
}
$$

This follows by induction on $|S|$, using the inductive definition of $\phi ^S$ and equation (\ref{bst}). Since $\phi ^S$ is invertible, we get a natural isomorphism as wanted.

For the first and second axioms of a transformation \cite[Section 1.2]{leinster} we need to show
$$\entrymodifiers={+!!<0pt,\fontdimen22\textfont2>}
\xymatrix{
A_S \ar[d]_{f_S} \ar[rr]^-{a^S} \drtwocell<\omit>{\quad \phi ^S} &&**[r]\Bin{S}A_i \ar[dl]^{\Bin{S}f_i}\ar[dd]^{\Bin{S}g_i\ast f_i}
&&& A_S \ar[dd]_{g_S\ast f_S} \ar[r]^-{a^S} \ddrtwocell<\omit>{\qquad (\gamma \diamond \phi ) ^S} &**[r]\Bin{S}A_i \ar[dd]^{\Bin{S}g_i\ast f_i}\\
A'_S \ar[d]_{g_S} \ar[r]^-{a'^S} \drtwocell<\omit>{\quad \gamma ^S} &**[r]\Bin{S}A'_i \ar[dr]^{\Bin{S}g_i} &&=\\
A''_S \ar[rr]_-{a''^S}&&**[r]\Bin{S}A''_i
&&& A''_S \ar[r]_-{a''^S}&**[r]\Bin{S}A''_i
}
$$
and
$$\entrymodifiers={+!!<0pt,\fontdimen22\textfont2>}
\xymatrixcolsep{4pc}\xymatrixrowsep{4pc}\xymatrix{
A_S \ar[d]_{\I{A_S}} \ar[r]^-{a^S} \drtwocell<\omit>{<2>\iota ^S \qquad \quad} &**[r]\Bin{S}A_i \dtwocell<5>_{\Bin{S}\I{A_i}\qquad}^{\qquad\I{\Bin{S}A_i}} \ar@{}[drr]|{=}
&& A_S \ar[d]_{\I{A_S}} \ar[r]^-{a^S} \ar[dr]_-{a^S} \druppertwocell<\omit>{<-4>l} \drlowertwocell<\omit>{<4>r^{-1}} &**[r]\Bin{S}A_i \ar[d]^{\I{\Bin{S}A_i}}\\
A_S \ar[r]_-{a^S}&**[r]\Bin{S}A_i
&& A_S \ar[r]_-{a^S}&**[r]\Bin{S}A_i
.
}
$$

The first one is straightforward using induction on $|S|$ and the definition of $\gamma \diamond \phi$. The second one holds again by induction on $|S|$, the fact that $\iota$ is a composition of instances of $\bp ^0$, $r$ and $l$, and the coherence conditions of $\bp ^0$ with respect to $r$ and $l$.

Hence we have a natural equivalence between $Id_{\gC{n}}$ and $i_n \circ p_n$. We conclude that the bicategories $\gC{n}$ and $\C ^{\times n}$ are equivalent, making $\widehat{\C}$ into a special $\Gamma$-bicategory.

\section{The construction of $\widetilde{\D}$}\label{constructionfat}

In this section we give the alternative version of the $\Gamma$-category for a permutative category $\D$ enriched over simplicial categories.

We need a preliminary observation. Let $\un{Hom}(X,Y)$ denote the internal hom in simplicial spaces. If $X$ is (the nerve of) a simplicial category, then there is a unique composition of paths:
$$\odot: \un{Hom}(\Delta ^1, X)\underset{X}\times \un{Hom}(\Delta ^1, X) \rightarrow \un{Hom}(\Delta ^1, X).$$ The uniqueness of the filling condition for all inner horns implies the uniqueness of filling for the spine, which in turn proves the following useful lemma.

\begin{lem}\label{multicell}
Composition of paths is associative. In particular, there is a unique map
$$\un{Hom}(\Delta ^1, X)\underset{X}\times \cdots  \underset{X}\times\un{Hom}(\Delta ^1, X) \rightarrow \un{Hom}(\Delta ^1, X).$$
\end{lem}

We will build a $\Gamma$-category $\widetilde{\D}$ which will turn out to be equivalent to the standard $\Gamma$-category construction $\widehat{\D}$. For the construction of $\fat{n}$ we mimic that of bicategories in Section \ref{proofgamma}.

\begin{cons}
  We build the $\Gamma$-category $\widetilde{\D}$ as follows. The objects are given by $\{ A_S, a _{S,T}\}$, with $S$ and $T$ as above; $A_S \in Ob \D$ and $a_{S,T} : A_{S\cup T}\rightarrow A_S \boxplus A_T$ is an invertible morphism, that is, a 0-simplex in $\D(A_{S\cup T}, A_S\bp A_T)$. We require conditions \ref{axiomobj1}-\ref{axiomobj4} in the proof of Theorem \ref{gamma} to hold. We note that the objects in $\widetilde{\D}(\underline{n})$ are the same as the objects in the usual $\Gamma$-category construction $\widehat{\D}(\underline{n})$.

Given two objects $\ob {A}{a}$ and $\ob {B}{b}$, the simplicial space of morphisms between them is defined as a subspace of
$$\prod _S \D(A_S, B_S) \times \prod _{S,T}\Hom{A_{\st}}{\bst}.$$

Let $X$ be defined by the pullback
\begin{equation}\label{pullback}
\xymatrix{
X \ar[r]\ar[d]& **[r]\prod _{S,T} ^{\phantom{X}} \Hom{A_{\st}}{\bst}\ar[d]^-{(d^{1\ast}, d^{0\ast})} \\
\prod _S \D(A_S, B_S) \ar[r]_-{\prod _{S,T}(a^{\ast},b_{\ast})}& **[r]\prod _{S,T}\D(A_{\st},\bst)^{\times 2} ,
}
\end{equation}
where the lower horizontal map corresponds to $(a^{\ast} _{S,T}, (b_{S,T}) _{\ast})$ for the $S,T$ component.

Let $Y$ be the equalizer
\begin{equation}\label{twist}
\entrymodifiers={+!!<0pt,\fontdimen22\textfont2>}
\xymatrix{
Y \ar[r] &X \ar@<1ex>[r]^-{p_1} \ar@<-1ex>[r]_-{p_2} &\prod_{S,T} \Hom{A_{\st}}{\bst}
},
\end{equation}
where $p_1$ is just the projection and $p_2$ is the projection composed with $\tau _{\ast}$. This reproduces condition (\ref{phits}) in this setting.

We define $\widetilde{\D}(\underline{n})(\ob{A}{a},\ob{B}{b}):=Z,$ where $Z$ is the equalizer
\begin{equation}\label{tripart}
\entrymodifiers={+!!<0pt,\fontdimen22\textfont2>}
\xymatrix{
Z \ar[r] &Y \ar@<1ex>[r]^-{q_1} \ar@<-1ex>[r]_-{q_2} &\prod_{S,T,U} \Hom{A_{\stu}}{\bstu}
},
\end{equation}
where $q_1$ and $q_2$ are as defined below. This construction mimics the condition of equation (\ref{phistu}).

Let $f$ be the following composition
\begin{equation*}
 \xymatrix{
  \D(A_U,B_U)\times \Hom{A_{S\cup T}}{\bst}\ar[d]_{\tilde{\op} \times \Id}\\
  \txt{$\Homdd{A_{\st}}{\bst}{A_{\st} \op  A_U}{\bstu}$\\
   $\times$\\
   $\Hom{A_{\st}}{\bst}$} \ar[d]_{\circ}\\
  \Hom{A_{\st}\! \op \! A_U}{\bstu}\ar[d]_{a^{\ast}}\\
  \Hom{A_{\stu}}{\bstu},
 }
\end{equation*}
where $\tilde{\op}, \circ, a$ denote the adjoint of $\op$, composition of maps, $a_{\st, U}$, respectively.

We also have the map
\begin{equation*}
 \xymatrix{
  \Hom{A_{\stu}}{B_{\st}\op B_U}\ar[d]_{b_{\ast}}\\
  \Hom{A_{\stu}}{\bstu},
 }
\end{equation*}

Note that the pullback conditions on $X$ imply that $d^{0\ast} f$ is equal to $d^{1\ast}b_{\ast}$, thus, putting the two maps $f$ and $b_{\ast}$ together we get a map from $Y$ to
$$\Hom{A_{\stu}}{\bstu}\!\!\!\!\!\!\!\!\!\!\!\!\!\!\!\!\underset{\D(\!{A^{\phantom{X}}_{\stu}},{\bstu}\!)}\times \!\!\!\!\!\!\!\!\!\!\!\!\!\!\!\!\Hom{A_{\stu}}{\bstu}.$$

We can further compose with $\odot$, getting a map to $\Hom{A_{\stu}}{\bstu}$.

Finally, we take the product over all $(S,T,U)$ to get
$$q_1: Y\longrightarrow \prod _{S,T,U} \Hom{A_{\stu}}{\bstu}.$$

We define $q_2$ similarly, starting with
$$\D(A_S,B_S)\times \Hom{A_{\tu}}{\btu}$$
instead.

We now show that the collection of objects described above with the simplicial spaces of morphisms $\widetilde{\D}(\un{n})$ form a category enriched over simplicial spaces.

Given objects $A=\ob{A}{a}$, $B=\ob{B}{b}$, $C=\ob{C}{c}$, we define a composition map
$$\fat{n}(B,C)\times \fat{n}(A,B)\rightarrow \fat{n}(A,C)$$
as follows: On one hand, we have a map given by the composition maps in $\D$:
$$\prod _S\D(A_S,B_S)\times \prod_S\D(B_S,C_S) \overset{\bullet}\longrightarrow \prod _S\D(A_S,C_S).$$

Given $(S,T)$, let $g$ be the following composition
\begin{equation*}
 \xymatrix{
  \D(B_{S},C_{S})\times \D(B_{T},C_{T})\times\Hom{A_{\st}}{\bst}\ar[d]_{\op \times \Id}\\
  \D({\bst},\cst)\times \Hom{A_{\st}}{\bst}\ar[d]_{\tilde{\bullet}\times \Id}\\
  \txt{$\Homdd{A_{\st}}{\bst}{A_{\st}}{\cst}$\\
  $\times$\\
  $\Hom{A_{\st}}{\bst}$}\ar[d]_{\circ}\\
  \Hom{A_{\st}}{\cst}.
 }
\end{equation*}

Similarly, we let $h$ be
\begin{equation*}
 \xymatrix{
  \D(A_{\st},B_{\st})\times \Hom{B_{\st}}{\cst}\ar[d]_{\tilde{\bullet}\times \Id}\\
  \txt{$\Homdd{B_{\st}}{\cst}{A_{\st}}{\cst}$\\
  $\times$\\
   $\Hom{B_{\st}}{\cst}$}\ar[d]_{\circ}\\
  \Hom{A_{\st}}{\cst}.
 }
\end{equation*}

One can check that when we restrict the source of these maps to the pullback diagram (\ref{pullback}),  the condition $d^{0\ast}g=d^{1\ast}h$ is satisfied. Hence we can put $g$ and $h$ together to get a map to
$$\Hom{A_{\st}}{\cst}\underset{\D({A_{\st}},{\cst})}\times \Hom{A_{\st}}{\cst}$$
that we can later compose with $\odot$ we get a map to
$$\Hom{A_{\st}}{\cst}.$$

Taking the product of these maps and the composition maps over $S,T$, we get a map
$$\fat{n}(B,C)\times \fat{n}(A,B) \overset{k}\longrightarrow \prod _S \D(A_S,C_S) \times \prod _{S,T} \Hom{A_{\st}}{\cst}.$$
We want to show now that the image of $k$ is contained in $\fat{n}(A,C)$.

On the $\Hom{A_{\st}}{\cst}$-component, $d^{1\ast}k=d^{1\ast}g$, which is equal to $a^{\ast}$ on $\D(A_S,C_S)\! \times \!\D(A_T,C_T)$ by the pullback (\ref{pullback}) on the space $\fat{n}(A,B)$.
%

Similarly for $d^{0\ast}k$, we show it is the same as $c_{S,T\ast}$, thus showing we land in the pullback (\ref{pullback}) for $\fat{n}(A,C)$.

To show that conditions (\ref{twist}) and (\ref{tripart}) hold we use the fact that they hold for $\fat{n}(A,B)$ and $\fat{n}(B,C)$ together with Lemma \ref{multicell}. We conclude that the composition is well defined. The fact that composition in $\D$ is associative and Lemma \ref{multicell} imply that the composition is associative.

The identity of $\ob{A}{a}$ is the 0-simplex given by the collections of $id _{A_S}$ together with the constant path at $a_{S,T} \in \D(A_{\st}, A_S\op A_T)$.

We thus get a category $\fat{n}$. We can extend this construction to a $\Gamma$-category in the usual way. Let $\theta : \underline{n} \rightarrow {m}$ be a morphism in $\Gamma$. We construct the functor
$$\theta _{\ast}: \fat{n}\longrightarrow \fat{m}$$
as follows.

We send the object $\ob{A}{a}$ of $\fat{n}$ to the object $\ob{A^{\theta}}{a^{\theta}}$ of $\fat{m}$, where $A_S^{\theta}=A_{\theta^{-1}S}$ and $a_{S,T}^{\theta}=a_{\theta^{-1}S, \theta^{-1}T}$.

For morphisms we construct a map
$$\fat{n}(A,B)\longrightarrow \fat{m}(A^\theta, B^\theta)$$
sending the projections of the $\theta^{-1}S$, $(\theta^{-1}S,\theta^{-1}T)$-components to the $S$, $(S,T)$-components, respectively.

\end{cons}

Consider the following diagram of functors
  \begin{equation*}
   \xymatrix{
    \widetilde{\D}(\underline{n})  \ar[rr]^{q}&&\D^{\times n} \ar@/^1pc/ @{.>}[dl]^{j}\\
    &\widehat{\D}(\underline{n}) \ar[ul]^{i} \ar[ur]^{p}
   }
  \end{equation*}
Here $q$ and $p$ are the projections onto the $(\{1\},\cdots, \{n\})$-components and $i$ is the inclusion. It is clear that $qi=p$. On the other hand, $j$ is the usual inverse equivalence for $p$; we will describe $j$ explicitly below.

\begin{prop}\label{ddd}
 The functors $i$ and $q$ are weak equivalences of categories (that is, they induce equivalences at the level of classifying spaces).
\end{prop}

\begin{cor}
 The $\Gamma$-category $\widetilde{\D}$ is special and it is levelwise equivalent to $\widehat{\D}$, via the map $i$.
\end{cor}

\begin{proof}[Proof of Proposition \ref{ddd}]
 The proof will proceed as follows. Since $i$ is the identity on objects and $q$ is surjective on objects, it is enough to prove that both $i$ and $q$ induce weak equivalences for the simplicial spaces of morphisms.

 We will first recall the definition of $j$ in \cite[Lemma 2.2]{SS}.

Given an object $(A_1,\cdots, A_n)$ in $\D ^{\times n}$, we let
 $$j(A_1,\cdots , A_n)=\{ \bigoplus _{i\in S} A_i, e_{S,T} \}$$
 where the sum is taken in the order of the indices of $S \subset \un{n}$. The morphism $e_{S,T}$ is the uniquely determined isomorphism from $A_{\st}$ to $A_S\op A_T$ given by composition of instances of $\tau$.

 For morphisms, we let
 $$j(f_1, \cdots f_n)=\{ \bigoplus _{i\in S} f_i \}.$$

 We then have that $pj=\Id$ and that there is a natural isomorphism $\lambda: \Id \rightarrow jp$ given on the object $\ob{A}{a}$ by the composition
 $$A_S \rightarrow A_{\{i_1\}} \op A_{S-i_1} \rightarrow \cdots A_{\{i_1\}} \op \cdots \op A_{\{i_k\}}$$
 of the corresponding $a$'s.

 Given an object $A=\ob{A}{a}$ in $\thin{n}$ (and thus, in $\fat{n}$), let $\dot{A}=\{ \bigoplus _{i\in S} A_i, e^A _{S,T} \} $ denote its image under $jq$.

 Given a pair $(A,B)$ of objects in $\thin{n}$, consider the following diagram:
 \begin{equation*}
  \xymatrix{
   \thin{n}(A,B) \ar[r]^-i &\fat{n}(A,B)\ar[d]^-q\\
   \thin{n}(\dot{A},\dot{B})\ar[u]^-{\lambda}_-{\cong} & \prod _i \D(A_i,B_i)\ar[l]^-j_-{\cong}
  }
\end{equation*}
The map $\lambda$ above is gotten by pre- and post-composition with the natural isomorphism $\lambda$ already defined. We note that $\lambda jqi=\Id$ and $qi\lambda j=\Id$.

Thus, if we show that $i\lambda jq $ is homotopic to the identity we will have shown that $\lambda j q$ and  $i\lambda j$ are homotopy inverses for $i$ and $q$ respectively, giving us the result we want.

To build a homotopy from $f=i\lambda jq $ to the identity, we show that there exists a map $\phi$ making the diagram below commute. For ease of notation we let $Z=\fat{n}(A, B)$.

\begin{equation*}
 \xymatrix{
  & Z\\
  Z \ar[ur]^{f} \ar@{-->}[r]^{\phi} \ar[dr]_{\Id} & **[r]\underline{Hom}(\Delta ^1, Z) \ar[u]_{d^{1\ast}} \ar[d]^{d^{0\ast}}\\
  &Z
 }
\end{equation*}

More precisely, we will construct a map
\begin{equation*}
 \xymatrix{
  \prod _S \D\!(A_S, B_S) \times \prod _{S,T} \Hom{A_{\st}}{\bst} \ar[d]\\
  \prod_S \Hom{A_S}{B_S} \times \prod _{S,T}\underline{Hom}(\Delta ^1 \times \Delta ^1, \D\!(A_{\st},\bst)),
 } 
\end{equation*}
and show that when restricted to $Z$ it lands on $\underline{Hom}(\Delta ^1, Z)$.

We will first construct the map $\phi _S$ to the component $\Hom{A_S}{B_S}$. Say $S=\{ i_1< \cdots < i_k\}$. Let $S_j=\{i_j, \cdots ,i_k\}$.

We consider the map $g_j$
\begin{equation*}
 \xymatrix{
  \prod _ {l=1}^{j-1}\D (A_{i_l},B_{i_l})\times \Hom{A_{S_j}}{B_{i_j}\op B_{S_{j+1}}} \ar[d]_{\oplus}\\
  \D(\bigoplus _{l=1}^k A_{i_l}, \bigoplus _{l=1}^k B_{i_l}) \times \Hom{A_{S_j}}{B_{i_j}\op B_{S_{j+1}}} \ar[d]_{\tilde{\op}}\\
  \Hom{A_{i_1}\op \cdots \op A_{i_{j-1}}\op A_{S_j}}{B_{i_1}\op \cdots \op B_{i_j}\op B_{S_{j+1}}}\ar[d]_{(a^{\ast})_{\ast}((b^{-1})_{\ast})_{\ast}}\\
  \Hom{A_S}{B_S}.
 }
\end{equation*}

Here $a$ is the composition of the instances of $a_{S,T}$ that take $A_S$ to $A_{i_1}\op \cdots \op A_{i_{j-1}}\op A_{S_j}$, and similarly for $b$ taking $B_S$ into $B_{i_1}\op \cdots \op B_{i_j}\op B_{S_{j+1}}$.

When we restrict to $Z$, $d^1 g_j=d^0 g_{j+1}$ and thus we get a map into
$$\underline{Hom}(Spine _{|S|},\D(A_S,B_S)).$$

By Lemma \ref{multicell}, this extends to $\Hom{A_S}{B_S}.$ The result is the map $\phi _S$.

From the construction it is clear that $d^1 \phi _S =f$ and $d^0\phi _S=\Id$ on this component.

Now, we look at the component $\underline{Hom}(\Delta ^1 \times \Delta ^1, \D(A_{\st},\bst))$.

Let $g$ be the composition
\begin{equation*}
 \xymatrix{
  Z \ar[d]_{\phi _S \times \phi _T}\\
  \Hom{A_S}{B_S}\times \Hom{A_T}{B_T} \ar[d]_{\op _{\ast}}\\
  \Hom{A_S\op A_T}{\bst}\ar[d]_{(a^{\ast})_{\ast}}\\
  \Hom{A_{\st}}{\bst}.
 }
\end{equation*}

Note that $g$ and the projection $Z\rightarrow \Hom{A_{\st}}{\bst}$ give composable paths. By Lemma \ref{multicell} and the given conditions on $Z$, the composed path is equal to the map $h$:
\begin{equation*}
 \xymatrix{
  Z\ar[d]_{\phi _{\st}}\\
  \Hom{A_{\st}}{B_{\st}}\ar[d]_{b_{\ast}}\\
  \Hom{A_{\st}}{\bst}.
 }
\end{equation*}

We can put these maps together and think of them as giving a map into
$$\underline{Hom}(\Delta ^1 \times \Delta ^1, \D(A_{\st},\bst)),$$
since the relevant paths are equal:
$$\xymatrix{
\bullet \ar@{-}[r] |-{\SelectTips{cm}{}\object@{>}}^-{proj}  & \bullet \\
\bullet \ar@{-}[r] |-{\SelectTips{cm}{}\object@{>}}_-{id} \ar@{-}[u] |-{\SelectTips{cm}{}\object@{>}}^-{g}
& \bullet \ar@{-}[u] |-{\SelectTips{cm}{}\object@{>}}_-{h}
}$$

Using the conditions on $Z$, it is easy to check that this construction yields the desired map
$$Z \rightarrow \un{Hom}(\Delta ^1, Z)$$
which restricts to the identity and $f$ at each end and gives the homotopy we wanted.
\end{proof}

\section{Proof of Theorem \ref{mton}}\label{proofmton}

In this section we prove Theorem \ref{mton} which relates the $\Gamma$-bicategory construction of Theorem \ref{gamma} and the construction of $\widetilde{\D}$ from the previous section.

\begin{proof}[Proof of Theorem \ref{mton}]
  We will build a map of bisimplicial spaces from the levelwise nerve of the simplicial category $S\gC{n}$ to $N\gD{n}$. Since $N\C(A,B)$ maps into $\D(A,B)$ we can think of 1-morphisms in $\C$ as morphisms in $\D$. Thus we can think of objects in $\gC{n}$ as objects in $\gD{n}$.

 The $(0,-)$ simplicial spaces for $NS\gC{n}$ and $N\gD{n}$ are given by the objects of $\gC{n}$ and $\gD{n}$, respectively. The desired map at the $(0,-)$ level is obtained by the identification of objects of $\gC{n}$ as objects of $\gD{n}$.

 Recall that a 1-morphism in $\gC{n}$ is given by a collection $\ob{f}{\phi}$ of 1-morphisms and 2-morphisms in $\Mod$. By considering the maps $N\C(A,B)\rightarrow \D(A,B)$, we can think of $\phi _{S,T}$ as a 0-simplex in $\Hom{A_{\st}}{\bst}$. As noted throughout the construction,  the restrictions on $\gD{n}(A,B)$ reflect the coherence axioms for $\ob{f}{\phi}$, so in general, we can think of a 1-morphism in $\gC{n}$ as a 0-simplex in $\gD{n}(A,B)$. Similarly, we can think of a 2-morphism in $\gC{n}$ as a 1-simplex.

 We can thus construct a map
 $$NS\gC{n}_{p,q} \rightarrow N\gD{n}_{p,q}$$
 as follows.

Recall that a $(p,q)$-simplex in $NS\gC{n}$ is given by a collection $\{ A_i \} _{i=0} ^p $ of objects in $\gC{n}$, diagrams of the form
\begin{equation*}
\xymatrix{
& A_j \ar[dr]^{f_{jk}^l} &\\
A_i \ar[ur]^{f_{ij}^l} \ar[rr]_{f_{ik}^l}\rrtwocell<\omit>{<-2.5>\quad \varphi _{ijk}^l}  &&A_k,
}
\text{\quad for all }0\leq i < j < k \leq p, 0\leq l \leq q
\end{equation*}
 subject to the coherence conditions in (\ref{conditions}), together with coherent 2-morphisms $f^{l}_{ij}\Rightarrow {f^{l+1}_{ij}}$.

 We can map this to
 $$\coprod _{A_0, \dots, A_p} \gD{n}(A_0, A_1) \times \cdots \times \gD{n}(A_{p-1}, A_p)$$
 by projecting the $(i,i+1)$-entries and using the identification of 1-morphisms and 2-morphisms in $\gC{n}$ as 0-simplices and 1-simplices of $\gD{n}$ described above.

 It is clear that this map extends to a map of $\Gamma$-spaces.
\end{proof}


\appendix\section{Classifying spaces of bicategories}\label{classifying}

Categories are closely related to spaces through the classifying space construction. To every category we can assign a space. This assignment gives a functor that is part of a Quillen equivalence between Thomason's model structure on the category of small categories \cite{thomcat} and the usual model structure in the category of topological spaces.

The same can be done with bicategories. In fact, there are many distinct constructions of the classifying space of a bicategory (\cite{CCG}). All these constructions give equivalent spaces, in the non-enriched case. Here we  describe the version of the nerve that was used in the preceding sections.

Lack and Paoli \cite{LP} introduce a version of a nerve of non-enriched bicategories that gives rise to a simplicial object in the category of small (non-enriched) categories. The construction can be extended to the enriched case, giving a simplicial object in $Cat$, where $Cat$ is the category of small categories enriched over simplicial sets. This construction is closely related to the bar construction for monoidal categories defined in \cite{BDR}, as we pointed out. This nerve is called \emph{2-nerve} in \cite{LP} and \emph{Segal nerve} in \cite{CCG}.

\begin{defn}
 Let $\mathcal{C}$ be a bicategory. The \emph{Segal nerve} $\bN\mathcal{C}$ is the simplicial object in $Cat$ given by normal homomorphisms, that is,
 $$\bN_n\mathcal{C}= \underline{NorHom}([n],\mathcal{C}),$$
 where the objects are normal pseudofunctors and the morphisms are \emph{icons}.
\end{defn}

Normal pseudofunctors are those for which the identity natural isomorphism is the identity. An \emph{icon} (Identity Component Oplax Natural transformation) is an oplax natural transformation (see \cite[Section 1.2]{leinster}) such that the map $\eta _A: \F A \rightarrow \G A$ is the identity, so in particular, we require that $\F A=\G A$. We now unravel the definition above.

An object of $\bN_n \mathcal{C}$ is given then by a collection of diagrams
\begin{equation}\label{simplices}
\xymatrix{
& C_j \ar[dr]^{f_{jk}} &\\
C_i \ar[ur]^{f_{ij}} \ar[rr]_{f_{ik}}\rrtwocell<\omit>{<-2.5>\quad \varphi _{ijk}}  &&C_k,
}
\text{\quad for all }0\leq i < j < k \leq n,
\end{equation}
where $\varphi _{ijk}$ is an invertible 2-morphism. This collection must satisfy the following coherence condition for all $0\leq i <j<k<l \leq n$:

\begin{equation}\label{conditions}
\xymatrixrowsep{.5pc}\xymatrixcolsep{.5pc}\xymatrix{
   &&C_j \ar[ddrr]& &  & &  & & & &C_j \ar[ddrr] \ar[dddd]\\
  &&\ddrtwocell<\omit>{<3>\quad \varphi _{ikl}} &&&&&&\ddrrtwocell<\omit>{\quad \varphi _{ijl}}&&\ddtwocell<\omit>{<-2.5>\varphi _{jkl}}& \\
C_i \ar[uurr] \ar[ddrr] \ar[rrrr] &\rrtwocell\omit{<-2.5>\quad \varphi _{ijk}} & &  &C_k \ar[ddll]& &= &&C_i \ar[uurr] \ar[ddrr]&&&   &C_k . \ar[ddll]\\
&&&&&&&&&&&\\
   &&C_l& &  & &  & &&  &C_l\\
}
\end{equation}

Given objects $\{ C_i, f _{ij}, \varphi _{ijk}\}$ and $\{ C_i, f' _{ij}, \varphi ' _{ijk}\}$ (note that the collections of objects are equal), a morphism between them is given by a collection of 2-morphisms $\eta _{ij}: f_{ij}\Rightarrow f'_{ij}$ for $i\leq j$, such that some coherence conditions (\cite[Eq. (44)]{CCG}) are satisfied.

\begin{rmk}\label{barsegal}
We note that the bar construction for monoidal categories of \cite{BDR} is equal to the Segal nerve. More precisely, if $\fM$ is a monoidal category,  then the simplicial category $B\fM$  of \cite{BDR} is equal to $\bN\Sigma \fM$ (with a possible reordering of the indices).
\end{rmk}

The Segal nerve is functorial with respect to normal pseudofunctors. It is the case that any pseudofunctor can be normalized (\cite[Prop. 5.2]{LP}).

It is clear from the definition that the Segal nerve preserves products.

\begin{defn}
 Let $\mathcal{C}$ be a bicategory. The \emph{classifying space} of $\mathcal{C}$ is the realization $|\bN\mathcal{C}|$.
\end{defn}

Let $\F, \G: \mathcal{C} \rightarrow \mathcal{D}$ be pseudofunctors, and $\eta : \F \rightarrow \G$ a transformation. As pointed out in the proof
of \cite[Prop. 7.1]{CCG}, these data gives rise to a pseudofunctor
$$\fH : \mathcal{C} \times \mathbf{1} \rightarrow \mathcal{D}$$
that restricts to $\F$ and $\G$ at $0$ and $1$. This pseudofunctor can be normalized, yielding the following result:

\begin{prop}\label{transhomo}
A transformation between pseudofunctors $\F, \G: \mathcal{C} \rightarrow \mathcal{D}$ gives rise to a homotopy between the maps
$$|\bN \F |, |\bN \G |: |\bN \mathcal{C} |\rightarrow |\bN\mathcal{D} |.$$
\end{prop}

\bibliography{ref}

\begin{thebibliography}{BDRR}

\bibitem[BDR]{BDR}
Nils~A. Baas, Bj{\o}rn~Ian Dundas, and John Rognes.
\newblock Two-vector bundles and forms of elliptic cohomology.
\newblock In {\em Topology, geometry and quantum field theory}, volume 308 of
  {\em London Math. Soc. Lecture Note Ser.}, pages 18--45. Cambridge Univ.
  Press, Cambridge, 2004.

\bibitem[BDRR]{BDRRst}
Nils~A. Baas, Bjorn~Ian Dundas, Birgit Richter, and John Rognes.
\newblock Stable bundles over rig categories.
\newblock Available as arXiv:0909.1742.

\bibitem[B{\'e}n]{benabou}
Jean B{\'e}nabou.
\newblock Introduction to bicategories.
\newblock In {\em Reports of the {M}idwest {C}ategory {S}eminar}, pages 1--77.
  Springer, Berlin, 1967.

\bibitem[CCG]{CCG}
P.~Carrasco, A.~M. Cegarra, and A.~R. Garzon.
\newblock Nerves and classifying spaces for bicategories, 2010.

\bibitem[EM]{EM}
A.~D. Elmendorf and M.~A. Mandell.
\newblock Rings, modules, and algebras in infinite loop space theory.
\newblock {\em Adv. Math.}, 205(1):163--228, 2006.

\bibitem[GPS]{GPS}
R.~Gordon, A.~J. Power, and Ross Street.
\newblock Coherence for tricategories.
\newblock {\em Mem. Amer. Math. Soc.}, 117(558):vi+81, 1995.

\bibitem[Gui]{guillou}
Bertrand Guillou.
\newblock Strictification of categories weakly enriched in symmetric monoidal
  categories.
\newblock Available as arXiv:0909.5270v1.

\bibitem[KS]{KS}
G.~M. Kelly and Ross Street.
\newblock Review of the elements of {$2$}-categories.
\newblock In {\em Category {S}eminar ({P}roc. {S}em., {S}ydney, 1972/1973)},
  pages 75--103. Lecture Notes in Math., Vol. 420. Springer, Berlin, 1974.

\bibitem[KV]{KV}
M.~M. Kapranov and V.~A. Voevodsky.
\newblock {$2$}-categories and {Z}amolodchikov tetrahedra equations.
\newblock In {\em Algebraic groups and their generalizations: quantum and
  infinite-dimensional methods ({U}niversity {P}ark, {PA}, 1991)}, volume~56 of
  {\em Proc. Sympos. Pure Math.}, pages 177--259. Amer. Math. Soc., Providence,
  RI, 1994.

\bibitem[Lap]{laplaza}
Miguel~L. Laplaza.
\newblock Coherence for distributivity.
\newblock In {\em Coherence in categories}, pages 29--65. Lecture Notes in
  Math., Vol. 281. Springer, Berlin, 1972.

\bibitem[Lei]{leinster}
Tom Leinster.
\newblock Basic bicategories.
\newblock Available as arXiv:math/9810017v1.

\bibitem[LP]{LP}
Stephen Lack and Simona Paoli.
\newblock 2-nerves for bicategories.
\newblock {\em $K$-Theory}, 38(2):153--175, 2008.

\bibitem[May1]{mayperm}
J.~P. May.
\newblock The spectra associated to permutative categories.
\newblock {\em Topology}, 17(3):225--228, 1978.

\bibitem[May2]{bipermnew}
J.~P. May.
\newblock The construction of {$E\sb \infty$} ring spaces from bipermutative
  categories.
\newblock In {\em New topological contexts for {G}alois theory and algebraic
  geometry ({BIRS} 2008)}, volume~16 of {\em Geom. Topol. Monogr.}, pages
  283--330. Geom. Topol. Publ., Coventry, 2009.

\bibitem[May3]{bipermold}
J.~Peter May.
\newblock {\em {$E\sb{\infty }$} ring spaces and {$E\sb{\infty }$} ring
  spectra}.
\newblock Lecture Notes in Mathematics, Vol. 577. Springer-Verlag, Berlin,
  1977.
\newblock With contributions by Frank Quinn, Nigel Ray, and J{\o}rgen
  Tornehave.

\bibitem[MLP]{MP}
Saunders Mac~Lane and Robert Par{\'e}.
\newblock Coherence for bicategories and indexed categories.
\newblock {\em J. Pure Appl. Algebra}, 37(1):59--80, 1985.

\bibitem[Seg]{segal}
Graeme Segal.
\newblock Categories and cohomology theories.
\newblock {\em Topology}, 13:293--312, 1974.

\bibitem[Shu]{shulman}
Michael Shulman.
\newblock Constructing symmetric monoidal bicategories.
\newblock Available as arXiv:1004.0993v1.

\bibitem[SP]{schommer}
Christopher Schommer-Pries.
\newblock The classification of two-dimensional extended topological field
  theories.
\newblock Available at {\tt
  http://sites.google.com/site/chrisschommerpriesmath/Home}.

\bibitem[SS]{SS}
Nobuo Shimada and Kazuhisa Shimakawa.
\newblock Delooping symmetric monoidal categories.
\newblock {\em Hiroshima Math. J.}, 9(3):627--645, 1979.

\bibitem[Str]{streetfib}
Ross Street.
\newblock Fibrations in bicategories.
\newblock {\em Cahiers Topologie G\'eom. Diff\'erentielle}, 21(2):111--160,
  1980.

\bibitem[Tho]{thomcat}
R.~W. Thomason.
\newblock Cat as a closed model category.
\newblock {\em Cahiers Topologie G\'eom. Diff\'erentielle}, 21(3):305--324,
  1980.

\end{thebibliography}
\bibliographystyle{alphanum}

\end{document}